\documentclass[reqno, oneside, 12pt]{amsart}

\usepackage[letterpaper]{geometry}
\usepackage{enumerate, hyperref,url,amssymb,amsmath,amsthm,amsxtra,mathtools,mathrsfs,calc,nccmath,color}

\usepackage{calc}
\usepackage{graphicx}
\usepackage{caption}
\usepackage{subcaption}

\newcommand{\Z}{\mathbb{Z}}
\newcommand{\Q}{\mathbb{Q}}

\newcommand{\C}{\mathbb{C}}

\let\temp\phi
\let\phi\varphi
\let\varphi\temp

\renewcommand{\(}{\left(}
\renewcommand{\)}{\right)}

\newcommand{\ord}{\operatorname{ord}}
\newcommand{\SL}{\operatorname{SL}}
\newcommand{\GL}{\operatorname{GL}}
\newcommand{\Frob}{\operatorname{Frob}}
\newcommand{\Sh}{\operatorname{Sh}}
\newcommand{\Gal}{\operatorname{Gal}}

\renewcommand{\sl}{\big|}
\newcommand{\sk}{\big|_k }

\renewcommand{\bar}[1]{\overline{#1}}

\newtheorem{theorem}{Theorem}[section]
\newtheorem{lemma}[theorem]{Lemma}
\newtheorem{definition}[theorem]{Definition}

\newtheorem{proposition}[theorem]{Proposition}

\theoremstyle{remark}
\newtheorem*{remark}{Remark}

\numberwithin{equation}{section}

\begin{document}


\title{	Congruence relations for $r$-colored partitions}

\date{\today}
\author{Robert Dicks}
\address{Department of Mathematics\\
University of Illinois\\
Urbana, IL 61801} 
\email{rdicks2@illinois.edu}

 
\begin{abstract}
Let $\ell \geq 5$ be prime. For the partition function $p(n)$ and $5 \leq \ell \leq 31$, Atkin found a number of examples of primes $Q \geq 5$ such that there exist congruences of the form 
$p(\ell Q^{3} n+\beta) \equiv 0 \pmod{\ell}.$
 Recently, Ahlgren, Allen, and Tang proved that there are 
infinitely many such congruences for every $\ell$.
In this paper, for a wide range of $c \in \mathbb{F}_{\ell}$, we prove congruences of the form $p(\ell Q^{3} n+\beta_{0}) \equiv c \cdot p(\ell Q n+\beta_{1}) \pmod{\ell}$ for infinitely many primes $Q$.
For a positive integer $r$, let $p_{r}(n)$ be the $r$-colored partition function. 
Our methods yield similar congruences for $p_{r}(n)$.
In particular, if $r$ is an odd positive integer 
for which $\ell > 5r+19$ and $2^{r+2} \not \equiv 2^{\pm 1} \pmod{\ell}$, then we show that there are infinitely many congruences of the form $p_{r}(\ell Q^{3}n+\beta) \equiv 0 \pmod{\ell}$. Our methods involve the theory of modular Galois representations.
 \end{abstract}


\maketitle

 
 \section{Introduction}
 Let $n \geq 1$ be an integer. A \emph{partition} of $n$ is any non-increasing sequence of positive integers whose sum is $n$. The partition function $p(n)$ counts the number of partitions of $n$. We agree that $p(0)=1$ and that $p(n)=0$ if $n \not \in \{0,1,2, \dots \}.$
  It is well-known that we have
  \begin{equation}\label{generatingfunction}
 \sum^{\infty}_{n=0}p(n)q^n= \prod^{\infty}_{n=1}(1-q^n)^{-1}.
 \end{equation}
 
  The study of the arithmetic properties of $p(n)$ has a long history. For example,
 Ramanujan \cite{Ramanujan} proved the following congruences for $p(n)$:
\begin{equation}\label{Ramanujancongruences}
 p(\ell n+\beta_{\ell}) \equiv 0 \pmod {\ell} \ \ \ \ \ \text{ for $\ell=5,7,11,$}
\end{equation}
 where $\beta_{\ell}:=\frac{1}{24} \pmod{\ell}$. 
  Since the work of Ramanujan, many authors found more examples of congruences for $p(n)$. 
 These take the form
 \begin{equation}\label{JarJaristhekeytoallofthis}
 p(\ell Q^{m}n+\beta) \equiv 0 \pmod{\ell},
 \end{equation}
 where $\ell \geq 5$ is prime, $Q$ is a prime distinct from $\ell$ and $m \geq 1$ is an integer.
 In his breakthrough work, Ono \cite{Kenbreakthrough} showed that for every $\ell \geq 5$, there are infinitely many primes $Q$ for which \eqref{JarJaristhekeytoallofthis} holds with $m=4$.
Such arithmetic phenomena also occur for the coefficients of a wide class of weakly holomorphic modular forms (see e.g. \cite{Treneer} and \cite{Treneer2}).
In contrast to this, Ahlgren and Boylan \cite{Ahlgren-BoylanInventiones} showed that if \eqref{JarJaristhekeytoallofthis} holds for $m=0$, then $\ell=5,7, \text{ or } 11$.
For $m=1,2$, Ahlgren, Beckwith, and Raum \cite{Scarcity} have shown that congruences of the form \eqref{JarJaristhekeytoallofthis} are scarce in a precise sense.

In the 1960s, Atkin discovered examples of congruences \eqref{JarJaristhekeytoallofthis} with $m=3$.
In recent work, Ahlgren, Allen, and Tang \cite{AAT} study the existence of such congruences.
These come from two families which we describe below.

For primes $13\leq \ell \leq 31$, Atkin \cite[eq. $52$]{Hestartedthis} gave examples of primes $Q$ for which
\begin{equation}\label{AtkintypeI}
p\(\frac{\ell Q^{2} n+1}{24}\) \equiv 0 \pmod{\ell} \ \ \ \text{ if } \(\frac{n}{Q}\)=\epsilon_{Q},
\end{equation}
where $\epsilon_{Q} \in \{\pm 1\}$.
By fixing $n$ in one of the allowable residue classes modulo $24Q$, we obtain a congruence \eqref{JarJaristhekeytoallofthis} with $m=3$.
We refer to these as type I congruences. 
For each of $\ell=5,7, \text{ and } 13$, Atkin \cite[Thm. $1$, $2$]{Hestartedthis} showed that for $Q \equiv -2 \pmod{\ell}$, we have
\begin{equation}\label{AtkintypeII}
p\(\frac{Q^{2}n+1}{24}\) \equiv 0 \pmod{\ell} \ \ \text{ if } \(\frac{-n}{\ell}\)=-1 \text{  and  } \(\frac{-n}{Q}\)=-1.
\end{equation}
We refer to these as type II congruences. By using modular Galois representations, Ahlgren, Allen and Tang proved that there are infinitely many type I congruences for every prime $\ell \geq 5$ and that there are infinitely many type II congruences for $\geq 17/24$ of the primes $\ell$.

For a positive integer $r$, define the \emph{$r$-colored partition function} by
\[
\sum^{\infty}_{n=0}p_{r}(n)q^{n}:= \prod^{\infty}_{n=1}(1-q^n)^{-r}=1+rq+\cdots.
\]
We agree that $p_{r}(n)=0$ if $n \not \in \{0,1,2,...\}$. Let $\ell \geq 5$ be prime.
Many congruences have been proven for $p_{r}(n)$ (see e.g. \cite[Thm. $1$]{AtkinRamanujan} and \cite[Thm. $2$]{Gordon}).
 In this paper, for a wide range of $c \in \mathbb{F}_{\ell}$, we prove congruences of the form $p_{r}(\ell Q^{3} n+\beta_{0}) \equiv c \cdot p_{r}(\ell Q n+\beta_{1}) \pmod{\ell}$ for infinitely many primes $Q$. In particular, when $c=0$, we get congruences analogous to \eqref{AtkintypeI}. 
Our first result relies on the assumption that $r$ and $\ell$ are \emph{compatible} 
 (we choose to delay the precise definition to Section $2$ for ease of exposition). Moreover, we will prove in Section~$5$ that $r$ and $\ell$ are compatible if we have $2^{r+2} \not \equiv 2^{\pm 1} \pmod{\ell}$ and $\ell > 5r+19$. We will also prove in Section $5$ that if $1 \leq r \leq 39$ and $\ell > r+4$, then $r$ and $\ell$ are compatible whenever $2^{r+2} \not \equiv 2^{\pm 1} \pmod{\ell}$.
Stating the result requires some notation. Define
\[
 \chi^{(r)}:= \begin{cases} 
\(\frac{-4}{\bullet}\) & \text{if } 3 \mid r, \\
\(\frac{12}{\bullet}\)& \text{if }  3\nmid r.
\end{cases}
\] 
Our first result states that we have type I congruences modulo $\ell$ for the functions $p_{r}(n)$ when $r$ and $\ell$ are compatible.
We state these theorems under the assumption that $r < \ell-4$. When $r=\ell-4$, the situation is much simpler and will be discussed in the last section.
\begin{theorem}\label{experimental main thm}
Suppose that $\ell \geq 5$ is prime, that $r < \ell-4$ is an odd positive integer
and that $r$ and $\ell$ are compatible. Then 
 there exists a positive density set $S$ of primes such that if $Q \in S$, then we have  $Q \equiv 1\pmod{\ell}$ and
 \[
 p_{r}\(\frac{\ell Q^{2}n+r}{24}\) \equiv 0 \pmod{\ell} \ \ \ \text{ if  } \ \ \ \(\frac{n}{Q}\)=\chi^{(r)}(Q)\(\frac{12}{Q}\)\(\frac{-1}{Q}\)^{\frac{\ell-r-2}{2}}.
 \]
\end{theorem} 
Thus, by choosing $n$ in any of the allowable residue classes modulo $24Q$, we obtain congruences of the form
\[
p_{r}(\ell Q^{3}n+\beta) \equiv 0 \pmod{\ell}.
\]
\begin{remark}
In the above theorem and the theorems which follow, our definition of density is that of natural density.
\end{remark}
\begin{remark}
When $r$ and $\ell$ are close, the modular form $f_{r,\ell}$ which captures the relevant values of $p_{r}(n)$ modulo $\ell$ has small weight. 
We illustrate this when $r=\ell-6$. In this case,
the form $f_{\ell-6,\ell}$ is in $S_{\frac{5}{2}}(1,\nu^{6\ell-1}_{\eta})$ (see Section $2$ for a definition of this space), which has dimension $0$ if $\ell \not \equiv 1 \pmod{4}$.
For such $\ell$, we have $f_{\ell-6,\ell}=0$, which implies that 
\[
p_{\ell-6}\(\frac{\ell  n+\ell-6}{24}\) \equiv 0 \pmod{\ell}.
\]
\end{remark}
We prove a more general result from which Theorem~\ref{experimental main thm} follows by setting $\alpha=1$.
\begin{theorem}\label{thm:main technical}
    Suppose that $\ell \geq 5$ is prime, that $r < \ell-4$ is an odd positive integer and that $r$ and $\ell$ are compatible. Suppose that $\alpha \not \equiv -2 \pmod{\ell}$ is an integer.
Then there exists a positive density set $S$ of primes such that if $Q \in S$, then $Q \equiv 1\pmod{\ell}$ and
 \[
 p_{r}\(\frac{\ell Q^{2}n+r}{24}\) \equiv (\alpha-1)\chi^{(r)}(Q)p_{r}\(\frac{\ell n+r}{24}\) \pmod{\ell} \ \text{ if } \ \ \(\frac{n}{Q}\)=\chi^{(r)}(Q)\(\frac{12}{Q}\)\(\frac{-1}{Q}\)^{\frac{\ell-r-2}{2}}.
 \]
 \end{theorem}
 Thus, by choosing $n$ in any of the allowable residue classes modulo $24Q$  
 we obtain congruences of the form
\[
p_{r}(\ell Q^{3}n+\beta_0) \equiv(\alpha-1)\chi^{(r)}(Q) p_{r}(\ell Q n+\beta_1) \pmod{\ell}.
\]
 \begin{remark}
If we make the further assumption that there is no congruence modulo any prime above $\ell$ between distinct newforms in $S\(-r\ell,\frac{\ell-r-1}{2}\)$ (see Section $6$ for the definition), Theorem~\ref{thm:main technical} also holds when $\alpha \equiv -2 \pmod{\ell}$.
 \end{remark}
The next result gives congruences analogous to $(1.4)$ and $(1.5)$ involving primes $Q$ in a different residue class modulo $\ell$. Our result does not rely on the assumption of compatibility.
\begin{theorem}\label{thm:arbitraryweight1}
Suppose that $r$ is an odd positive integer.
Suppose that $\ell \geq 5$ is prime, that $r < \ell-4$ and that
\begin{equation}\label{thatbluehedgehogagainofallplaces}
\text{ there exists an integer $a$ with $2^{a} \equiv -1 \pmod{\ell}$}.
\end{equation}
\begin{enumerate}
\item
There exists a positive density set $S$ of primes with the property that if $Q \in S$, then $Q \equiv -2 \pmod{\ell}$ and there exists $\epsilon_{Q} \in \{\pm 1\}$ such that
\[
p_{r}\(\frac{\ell Q^{2}n+r}{24}\) \equiv 0 \pmod{\ell} \ \ \  \text{ if } \\ \ \(\frac{n}{Q}\)= \epsilon_{Q}.
\]
\item
Suppose that all of the integers $-r \leq w<0$ with $w \equiv -r \pmod{24}$ satisfy $\(\frac{-rw}{\ell}\)=1$. Then there exists a positive density set $S$ of primes with the property that if $Q \in S$, then $Q \equiv -2 \pmod{\ell}$ and there exists $\epsilon_{Q} \in \{\pm 1\}$ such that
\[
p_{r}\(\frac{Q^{2}n+r}{24}\) \equiv 0 \pmod{\ell} \ \ \text{  if } \ \ \(\frac{-rn}{\ell}\)=-1 \ \ \ \text{ and }  \ \(\frac{n}{Q}\)= \epsilon_{Q}.
\]
\end{enumerate}
\end{theorem}

\begin{remark}
The value of $\epsilon_{Q}$ in $(1)$ and $(2)$ can be made explicit for any $r$ and $\ell$ using the arguments at the end of Section $6$.
\end{remark}
\begin{remark}
The condition on the integers $-r \leq w<0$ with $w \equiv -r\pmod{24}$ in $(2)$ is necessary to ensure that the form $g_{r,\ell}$ which we construct in Section $3$ is a cusp form.
We can use the same method to prove an analogous result for primes $\ell \geq 5$ with the property that $3^a \equiv -2 \pmod{\ell}$ for some integer $a$. However, 
we would need to add the condition that there are no congruences between distinct newforms in the spaces under consideration. 
\end{remark}
\begin{remark}
 If $r \leq 23$, then the condition on the integers $-r \leq w<0$ with $w \equiv -r\pmod{24}$ is always satisfied, since the only choice for $w$ is $-r$. It also holds for $r=25$ (since it is a square) and $r=27$ (since $\(\frac{3}{\ell}\)=\(\frac{27}{\ell}\)$).
For a fixed $r$, we can use the quantitative version of Dirichlet's theorem on primes in arithmetic progressions, quadratic reciprocity, and the Chinese remainder theorem to prove that the condition holds for a positive proportion of  the primes $\ell$.
\end{remark}
The organization of this paper is as follows. In Section $2$, we give some background on modular forms and Galois representations. In Section $3$, we construct for each $r$ and $\ell$ important modular forms whose coefficients capture the relevant values of $p_{r}(n)$ modulo $\ell$.
In the remaining sections (save for Section $7$), our main tool is the theory of modular Galois representations.
In Section $4$, we prove that a wide range of $r$ and $\ell$ are compatible.
 In Section $5$, we prove the main technical result which we will need in order to prove Theorem~\ref{thm:main technical}. 
 In Section $6$, we prove Theorems $1.2-1.3$.
 Finally, in Section $7$, we discuss congruences for $p_{r}(n)$ when $r=\ell-4$; in particular, we give a short proof of \cite[Thm. $2.1$, $(2)$]{Boylan}.
 \section{Background}
 We follow the exposition in \cite{Scarcity} and \cite{AAT}.
 Throughout, let $\ell \geq 5$ be prime and set $q:=e^{2 \pi i z}$. 
 Suppose that $k \in \frac{1}{2}\Z$ and that $N$ is a positive integer.
  For a function $f(z)$ on the upper half plane  and 
\[
 \gamma =\left(\begin{matrix}a & b \\c & d\end{matrix}\right) \in \GL_{2}^{+}(\Q),
\]
we have the weight $k$ slash operator 
\[
f(z)\sk \gamma := \det(\gamma)^{\frac{k}{2}}(cz+d)^{-k}f\(\frac{az+b}{cz+d}\).
\]

Let $A \subseteq \C$ be a subring.
If $\nu$ is a multiplier system on $\Gamma_0(N)$, 
we denote by $M_{k}(N,\nu,A)$ (resp. $S_{k}(N,\nu,A)$)
the space of modular forms (resp. cusp forms) 
of weight $k$ and multiplier 
$\nu$ on $\Gamma_0(N)$ whose Fourier coefficients are in $A$.
When $\nu=1$ or $A$ is the subring of algebraic numbers that are integral at all of the primes above $\ell$, we omit them from the notation.
Forms in these spaces satisfy the transformation law
\[
f \sk \gamma= \nu(\gamma)f \ \ \ \text{ for } \ \ \ \gamma = \left(\begin{matrix}a & b \\c & d\end{matrix}\right) \in \Gamma_0(N)
\]
and the appropriate conditions at the cusps of $\Gamma_0(N)$. 

If $k$ is even, we denote by $S^{\text{new}}_{k}(N,\C)$ the new subspace. Let
$S^{\text{new}}_{k}(N,A)$ be defined by $S_{k}(N,A)  \cap S^{\text{new}}_{k}(N,\C)$.
When $N$ is square-free, there is an Atkin-Lehner involution $W_{p}$ on $S_{k}(N,\C)$ for every prime divisor $p$ of $N$. Given a tuple $\epsilon=(\epsilon_{p})_{p \mid N}$ (where each $\epsilon_{p} \in \{\pm 1\}$), let $S^{\operatorname{new}}_{k}(N,\C,\epsilon)$ be the subspace of forms $f$ with $f\sl_{k} W_{p}=\epsilon_{p}f$ for $p \mid N$. If $\ell$ and $N$ are coprime, then we let $S^{\operatorname{new}}_{k}(N,\epsilon)$ be the subspace of $S^{\operatorname{new}}_{k}(N)$ attached to the tuple $\epsilon$ (for such $\ell$, the operator $W_{p}$ acts on $S^{\operatorname{new}}_{k}(N)$). 
We define the eta function by
\[
\eta(z):= q^{\frac{1}{24}}\prod_{n=1}^{\infty}(1-q^{n})
\]
and the theta function by
\[
\theta(z):= \sum_{n=-\infty}^{\infty}q^{n^2}.
\]
The eta function has a multiplier $\nu_{\eta}$ satisfying

\[
\eta(\gamma z)=\nu_{\eta}(\gamma)(cz+d)^{\frac{1}{2}}\eta(z), \ \ \ \ \ \gamma= \left(\begin{matrix}a & b \\c & d\end{matrix}\right) \in \SL_{2}(\Z);
\]
throughout, we choose the principal branch of the square root. For $c > 0$,
we have the formula 
  \cite[~$\mathsection$$4.1$]{Knopp}

\begin{equation}\label{etamultiplier}
\nu_{\eta}(\gamma)=
 \begin{cases} 
 \(\frac{d}{c}\)e\(\frac{1}{24}((a+d)c-bd(c^2-1)-3c )\) & \text{if } c \text{ is odd,} \\
\(\frac{c}{d}\)e\(\frac{1}{24}((a+d)c-bd(c^2-1)+3d-3-3cd)\) & \text{if } c \text{ is even}.
\end{cases}
\end{equation}
For the multiplier of the theta function we have
\[
\nu_{\theta}(\gamma):= (cz+d)^{-\frac{1}{2}}\frac{\theta(\gamma z)}{\theta(z)}=\(\frac{c}{d}\)\epsilon_{d}^{-1}, \ \ \ \ \ \gamma= \left(\begin{matrix}a & b \\c & d\end{matrix}\right) \in \Gamma_0(4),
\]
where
\[
\epsilon_{d}=
\begin{cases}
1 & \text{if } d \equiv 1 \pmod{4}, \\
i & \text{if } d \equiv 3 \pmod{4}.
\end{cases}
\]
For odd values of $d$, we have the formula
\begin{equation}\label{pain}
e\(\frac{1-d}{8}\)=\(\frac{2}{d}\)\epsilon_{d}.
\end{equation}
For $r \in \Z$, we have
\begin{equation}\label{had to tag this}
 M_{k}(1,\nu^{r}_{\eta})=\{0\} \ \ \ \text{if} \ \ \ 2k-r \not \equiv 0 \pmod{4}.
\end{equation}
If $f \in M_{k}(1,\nu^{r}_{\eta})$, then
$\eta^{-r}f$ is a weakly holomorphic form on $\SL_2(\Z)$.
This implies that each $f \in M_{k}(1,\nu^{r}_{\eta})$ has a Fourier expansion of the form
\begin{equation}\label{rmod24expansion}
f=\sum_{n \equiv r \pmod{24}} a(n)q^{\frac{n}{24}}.
\end{equation}
We have
$\eta^{-r}f \in M_{k-\frac{r}{2}}(1)$
when $0<r<24$.

We next recall the $U$ and $V$ operators. For a positive integer $m$, we define them on Fourier expansions by
\[
\(\sum_{n=1}^{\infty}a(n)q^{\frac{n}{24}}\)\sl U_{m}:= \sum_{n=1}^{\infty}a(mn)q^{\frac{n}{24}},
\]

\[
\(\sum_{n=1}^{\infty}a(n)q^{\frac{n}{24}}\)\sl V_{m}:= \sum_{n=1}^{\infty}a(n)q^{\frac{mn}{24}}.
\]
If $k \in \frac{1}{2}\Z \backslash \Z$,
then a computation using \eqref{etamultiplier}, \eqref{pain} and \eqref{had to tag this} implies that 
\begin{equation}\label{Shimura3nmidr}
f \in S_{k}(1,\nu^{r}_{\eta}) \implies f \sl V_{24} \in S_{k}\(576,\(\frac{12}{\bullet}\)\nu^{r}_{\theta}\) \ \ \ \text{ if $3 \nmid r$}
\end{equation}
 and
\begin{equation}\label{Shimura3midr}
f \in S_{k}(1,\nu^{r}_{\eta}) \implies f\sl V_{8} \in  S_{k}\(64,\nu^{r}_{\theta}\) \ \ \ \text{ if $3 \mid r$.}
\end{equation}

 If $k \in \frac{1}{2}\Z \backslash \Z$, then for each prime $Q \geq 3$ we have the Hecke operator
 \[
 T_{Q^{2}}: S_{k}(1,\nu^{r}_{\eta})\rightarrow S_{k}(1,\nu^{r}_{\eta}).
 \]
 For $f =\displaystyle \sum a(n)q^{\frac{n}{24}} \in S_{k}(1,\nu^{r}_{\eta})$ and $Q \geq 5$, we have
 \begin{equation}\label{Hecke1}
 f \sl T_{Q^{2}}=\sum \(a(Q^{2} n)+Q^{k-\frac{3}{2}}\(\frac{-1}{Q}\)^{k-\frac{1}{2}}\(\frac{12 n}{Q}\)a(n)+Q^{2k-2}a\(\frac{n}{Q^2}\)\)q^{\frac{n}{24}}.
 \end{equation}
(see e.g. \cite[Proposition $11$]{Yang}). Yang only states the result for $r$ such that $3 \nmid r$.  The result when $3 \mid r$ follows from \eqref{Shimura3midr} and \cite[Thm. $1.7$]{Shimura}. When $3 \mid r$, we also have \eqref{Hecke1} for $Q=3$.

For $k \in \frac{1}{2}\Z \backslash \Z$ such that $k \geq \frac{5}{2}$ and each odd squarefree $t$, if $3 \nmid r$ (resp. $3 \mid r$), then we have a Shimura lift on $S_{k}(1,\nu^{r}_{\eta})$ defined via \eqref{Shimura3nmidr} (resp. \eqref{Shimura3midr}) and the usual Shimura lift \cite{Shimura} on $S_{k}\(576,\(\frac{12}{\bullet}\)\nu^{r}_{\theta}\)$ (resp. $S_{k}\(64,\nu^{r}_{\theta}\)$).
Define
\[
 \psi^{(r)}:= \begin{cases} 
\(\frac{4}{\bullet}\) & \text{if } 3 \mid r, \\
\(\frac{12}{\bullet}\)& \text{if }  3\nmid r.
\end{cases}
\] 
and
\[
w_r:=
\begin{cases}
3 & \text{if } 3 \mid r, \\
1 & \text{if } 3 \nmid r.
\end{cases}
\]
The action on Fourier expansions is given by
\[
\text{Sh}_{t}\(\sum a(n)q^{\frac{n}{24}}\)=\sum A_{t}(n)q^{n},
\]
where
\[A_{t}(n)=\sum_{d \mid n}\(\frac{-1}{d}\)^{k-\frac{1}{2}}\psi^{(r)}(d)\(\frac{t}{d}\)d^{k-\frac{3}{2}}a\(\frac{w_{r}tn^{2}}{d^{2}}\).
\]
 We have (see e.g. \cite[$(2.13)$]{Scarcity})
\begin{equation}\label{f0iffShimura0}
f \equiv 0 \pmod{\ell} \iff \text{Sh}_{t}(f) \equiv 0 \pmod{\ell} \ \ \ \text{ for all squarefree $t$}.
\end{equation}
They only state their result for when $3 \nmid r$, but the same argument applies when $3 \mid r$.
From \cite[Thm. $1$, $2$]{Yang}, it follows that
\[
\text{Sh}_{t}:S_{k}(1,\nu^{r}_{\eta}) \rightarrow S^{\text{new}}_{2k-1}\(6,-\(\frac{8}{r}\),-\(\frac{12}{r}\)\) \otimes \(\frac{12}{\bullet}\) \ \ \text{if $3 \nmid r$}
\]
and
\[
\text{Sh}_{t}:S_{k}(1,\nu^{r}_{\eta}) \rightarrow S_{2k-1}^{\text{new}}\(2, \(\frac{8}{r}\)\) \otimes\(\frac{-4}{\bullet}\) \ \ \text{if $3\mid r$}
\]
(note that Yang uses $3r$ in place of $r$).
Moreover, we have
\[
\Sh_{t}\(f\sl T_{Q^{2}}\)=\(\Sh_{t}f\)\sl T_{Q},
\]
where $T_{Q}$ is the Hecke operator of index $Q$ on the integral weight space.

We now summarize some facts about modular Galois representations. See \cite{Hida} and \cite{Edixhoven} for more details. 
We begin with some notation. Recall that $\ell \geq 5$ is prime. Let $k$ be an even integer and $N \in \Z^{+}$ with $\ell \nmid N$. 
Let $\bar{\Q} \subseteq \C$ be the algebraic closure of $\Q$ in $\C$. If $p$ is prime, let $\bar{\Q}_{p}$ be a fixed algebraic closure of $\Q_{p}$ and fix an embedding $\iota_{p}:\bar{\Q} \hookrightarrow \bar{\Q}_{p}$. 
 The embedding $\iota_{\ell}$ allows us to view the coefficients of forms in $S_{k}(N)$ as elements of $\bar{\Q}_{\ell}$, and for each prime $p$,
the embedding $\iota_{p}$ allows us to view $G_{p}:=\Gal(\bar{\Q}_{p}/\Q_{p})$ as a subgroup of $G_{\Q}:=\Gal(\bar{\Q}/\Q)$.
 If $I_{p} \subseteq G_{p}$ is the inertia subgroup, we denote the coset of absolute Frobenius elements above $p$ in $G_{p}/I_{p}$ by $\text{Frob}_{p}$. 
 For any finite extension $K/\Q$, 
 denote by $\text{Frob}_{p} |_{K}$ the restrictions to $K$ of elements in $\text{Frob}_{p}$.
 For representations $\rho:G_{\Q} \rightarrow \GL_{2}(\bar{\Q}_{\ell})$ which are unramified at $p$ (which is to say that $I_{p} \subseteq \ker{\rho}$), the image of any $\sigma \in G_{p}$ under $\rho$ only depends on $\sigma I_{p} \in G_{p}/I_{p}$ (the same conclusion holds for representations of $\Gal(K/\Q)$ and $\sigma |_{K}$); for any such representation, we denote by $\rho(\text{Frob}_{p})$ the image of any element in $\text{Frob}_{p}$ under $\rho$ and make a similar definition for representations of 
 $\Gal(K/\Q)$ and $\operatorname{Frob}_{p} |_{K}$.

  We denote by $\chi:G_{\Q}\rightarrow \Z^{*}_{\ell}$ and $\omega:G_{\Q}\rightarrow \mathbb{F}^{*}_{\ell}$ the $\ell$-adic and mod $\ell$ cyclotomic characters, respectively.
 We let $\omega_{2},\omega'_{2}:I_{\ell}\rightarrow\mathbb{F}^{*}_{\ell^{2}}$ denote Serre's fundamental characters of level $2$ (see \cite[~$\mathsection$$2.1$]{DukeSerre}).
 Both characters have order $\ell^{2}-1$, and we have 
 $\omega^{\ell+1}_{2}=\omega'^{\ell+1}_{2}=\omega$.

The following theorem is due to Deligne, Fontaine, Langlands, Ribet, and Shimura (see \cite[Thm. 2.1]{AAT}). 
\begin{theorem}\label{bigGalThm}
Let $f=q+\sum_{n \geq 2}a(n)q^{n} \in S_{k}(N)$ be a normalized Hecke eigenform. There is a continuous irreducible representation $\rho_{f}:G_{\Q} \rightarrow \GL_{2}(\bar{\Q}_{\ell})$ with semisimple mod $\ell$ reduction $\bar{\rho}_{f}:G_{\Q} \rightarrow \GL_{2}(\bar{\mathbb{F}}_{\ell})$ satisfying the following properties.
\begin{enumerate}
\item
If $p \nmid \ell N$, then $\rho_{f}$ is unramified at $p$ and the characteristic polynomial of $\rho_{f}(\Frob_{p})$ is $X^2-\iota_{\ell}(a(p))X+p^{k-1}$. 
\item
If $q \mid N$ and $q^2 \nmid N$, then $\rho_{f} |_{I_{q}}$ is unipotent. 
In particular, the prime-to-$\ell$ Artin conductor $N(\bar{\rho}_{f})$ of $\bar{\rho}_{f}$ is not divisible by $q^{2}$.
If further $f$ is $q$-new, then we have
\[
\rho_{f}|_{G_{q}} \cong \left(\begin{matrix}\chi \psi & * \\0 & \psi \end{matrix}\right),
\]
where $\psi:G_{q} \rightarrow \bar{\Q}^{*}_{\ell}$ is the unramified character with $\psi({\Frob}_{q})=\iota_{\ell}(a_{q})$.
\item
Assume that $2 \leq k \leq \ell+1$. Then 
\begin{itemize}
\item
If $\iota_{\ell}(a_{\ell})$ is an $\ell$-adic unit, then $\rho_{f} |_{G_{\ell}}$ is reducible and we have
\[
\rho_{f}|_{I_{\ell}} \cong \left(\begin{matrix}\chi^{k-1} & * \\0 & 1 \end{matrix}\right).
\]
\item
If $\iota_{\ell}(a_{\ell})$ is not an $\ell$-adic unit, then $\bar{\rho}_{f}\sl_{G_{\ell}}$ is irreducible and $\bar{\rho}_{f}|_{I_{\ell}} \cong \omega^{k-1}_{2} \oplus \omega'^{(k-1)}_{2}$.
\end{itemize}
\end{enumerate}
\end{theorem}
\begin{remark}
The Galois representations depend on the choice of embedding $\iota_{\ell}:\bar{\Q} \hookrightarrow \bar{\Q}_{\ell}$, but we have suppressed this from the notation.
\end{remark}
We now define what it means for $r$ and $\ell$ to be compatible. Define
 \[
N_{r}:= \begin{cases} 
2 & \text{if } 3 \mid r, \\
6  & \text{if }  3\nmid r
\end{cases}
\] 
and 
 \[
\epsilon_{r}:= \begin{cases} 
\(\frac{8}{r\ell}\) & \text{if } 3 \mid r, \\
(-\(\frac{8}{r\ell}\),-\(\frac{12}{r\ell}\) ) & \text{if }  3\nmid r.
\end{cases}
\] 
\begin{definition}
Suppose that $\ell \geq 5$ is prime and that $r$ is an odd positive integer such that $r < \ell-4$.
We say that $r$ and $\ell$ are compatible if
for every newform $f \in S^{\operatorname{new}}_{\ell-r-2}(N_{r},\epsilon_{r})$ the image of $\bar{\rho}_{f}$ contains a conjugate of $\SL_{2}(\mathbb{F}_{\ell})$.
\end{definition}
We now discuss filtrations. Recall that $\ell \geq 5$ is prime.
 If $f \in M_{k}(1,\Z)$ is given by $f=\displaystyle \sum^{\infty}_{n=0} a(n)q^{n}$, then we define
\[
\bar{f}:= \sum^{\infty}_{n=0} \bar{a(n)}q^{n} \in \mathbb{F}_{\ell}[[q]]
\]
and
\[
w_{\ell}(\bar{f}):= \inf\{ k': \ \ \ \text{there exists $g \in M_{k'}(1,\Z)$ with $\bar{f}=\bar{g}$}\}.
\]
We also define 
\begin{equation}\label{RamanujanTheta}
\Theta:=\frac{1}{2\pi i}\frac{d}{dz}=q\frac{d}{dq}.
\end{equation}
We require the following facts (see e.g. \cite[$\mathsection$$2.2$]{Serre1} and \cite[$\mathsection$$1$]{Jochnowitz}).
\begin{lemma}\label{Serre}
Let $k$ and $k'$ be even integers. Let $f \in M_{k}(1,\Z)$ and $\ell \geq 5$ be prime. 
\begin{enumerate}
\item
We have
\[
w_{\ell}(f \sl U_{\ell}) \leq \ell+\frac{w_{\ell}(f)-1}{\ell}.
\]
\item
If $g \in M_{k'}(1,\Z)$ satisfies $\bar{f}=\bar{g}$, then we have $k \equiv k'\pmod{\ell-1}$.
\item
There exists a form $g \in S_{k+\ell+1}(1,\Z)$ with $g \equiv \Theta f \pmod {\ell}$.
\item
For $i \in \Z^{+}$, we have $w_{\ell}(f^{i})=iw_{\ell}(f)$.
\end{enumerate}
\end{lemma}

Finally, for an even integer $k$, denote the weight $k$ Eisenstein series on $\SL_2(\Z)$ by $E_{k}$.
 \section{The modular forms $f_{r,\ell}$ and $g_{r,\ell}$}
 Let 
 \[
 \Delta:= q\prod^{\infty}_{n=1}(1-q^{n})^{24}
 \]
 be the unique normalized cusp form of weight $12$ on $\SL_2(\Z)$.
 Set
 \[
 \delta_{\ell}:=\frac{\ell^{2}-1}{24}.
 \]
By studying the filtration $w_{\ell}(\Delta^{r\delta_{\ell}} \sl U_{\ell})$, we prove the following result.

\begin{proposition}\label{FILTRATION}
Suppose that $\ell \geq 5$ is prime. Assume that $r$ is an odd positive integer with $r \leq \ell-4$. 
Then there exists a modular form $f_{r,\ell} \in S_{\frac{\ell-r-1}{2}}(1,\nu^{-r\ell}_{\eta},\Z)$ with
\begin{equation}\label{THATBLUEHEDGEHOG}
f_{r,\ell} \equiv \sum p_{r}\(\ell n-r\delta_{\ell}\)q^{n-\frac{r\ell}{24}} \equiv \sum p_{r}\(\frac{\ell n+r}{24}\)q^{\frac{n}{24}}\pmod{\ell}.
\end{equation}
\end{proposition}
\begin{remark}
Boylan \cite[Theorem $1.3$]{Boylan} gave the complete list of pairs $(r,\ell)$ with $r \leq 47$ for which $f_{r,\ell} \equiv 0 \pmod{\ell}$.
\end{remark}
\begin{proof}
Note that 
\begin{equation}\label{relevantdeltaformula}
\Delta^{r\delta_{\ell}}=
q^{r\delta_{\ell}}\prod^{\infty}_{n=1}\frac{(1-q^{n})^{r\ell^2}}{(1-q^n)^r} \equiv 
\prod^{\infty}_{n=1}(1-q^{\ell n})^{r\ell} \cdot \sum^{\infty}_{n=0}p_{r}(n-r\delta_{\ell})q^n \pmod{\ell}.
\end{equation}
Thus, we have
\begin{equation}\label{cuttingbyetaisokay}
\Delta^{r\delta_{\ell}} \sl U_{\ell} \equiv \prod^{\infty}_{n=1}(1-q^n)^{r\ell} \cdot \sum^{\infty}_{n=0}p_{r}(\ell n -r\delta_{\ell})q^n \pmod{\ell}.
\end{equation}
Define
\[
k:=w_{\ell}\(\Delta^{r\delta_{\ell}}\sl U_{\ell}\).
\]
If $k= -\infty$, then the right hand side of \eqref{cuttingbyetaisokay} is congruent to $0 \pmod{\ell}$, so there is nothing to prove. Therefore, we assume that $k \neq -\infty$.
By definition, there exists a form $G_{r,\ell} \in M_{k}(1,\Z)$ with
\[
 G_{r,\ell} \equiv \Delta^{r\delta_{\ell}} \sl U_{\ell} \pmod{\ell}.
 \]
By $(4)$ of Lemma~\ref{Serre}, we see that $w_{\ell}(\Delta^{r\delta_{\ell}})=12r\delta_{\ell}$.
By $(1)$ of Lemma~\ref{Serre}, we have
\[
k \leq \frac{r+2}{2}\ell-\frac{r+2}{2\ell}.
\]
Since $\Delta^{r\delta_{\ell}} \sl U_{\ell} \equiv \Delta^{r\delta_{\ell}} \sl T_{\ell} \pmod{\ell}$ and $\Delta^{r\delta_{\ell}} \sl T_{\ell}$ has weight $r(\frac{\ell^{2}-1}{2}) \equiv 0 \pmod{\ell-1}$, we see by $(2)$ of Lemma~\ref{Serre} that 
\begin{equation}\label{mod l - 1}
k \equiv 0 \pmod{\ell-1}.
\end{equation}

If $r \leq \ell-4$, we have
\[
\frac{r+2}{2}\ell-\frac{r+2}{2\ell} < \frac{r+3}{2}(\ell-1)
\]
which implies by \eqref{mod l - 1} that
\begin{equation}\label{after mod l -1}
k \leq \frac{r+1}{2}(\ell-1).
\end{equation}
  Let $d=\dim(M_{k}(1,\Z)).$ The space $M_{k}(1,\Z)$ has a basis $\{f_0,...,f_d\}$ of forms such that
\begin{equation}\label{eqn:INTEGRALITY}
f_i(z)=q^{i}+O(q^{i+1}),  \ \ \ \ \ \ \ \ \ \ \ \ \ \ \ \ \ 0 \leq i \leq d-1.
\end{equation} 
After subtracting an appropriate integral linear combination of these basis elements from $G_{r,\ell}$, 
we may assume that 
$\ord_{\infty}(G_{r,\ell}) \geq \lceil r(\frac{\ell^{2}-1}{24\ell}) \rceil$. 
If $n \in \Z$ satisfies $\frac{r(\ell^{2}-1)}{24\ell} \leq n < \frac{r\ell}{24}$, then $r\ell^{2}-r\leq 24\ell n<r\ell^{2}$; this is a contradiction since $r \leq \ell-4$. This
implies that $\ord_{\infty}(G_{r,\ell}) \geq \lceil r(\frac{\ell^{2}-1}{24\ell}) \rceil \geq \frac{r\ell+1}{24} $, so 
\[
f_{r,\ell}:= \frac{G_{r,\ell}}{\eta^{r\ell}} \in S_{k-\frac{r\ell}{2}}(1,\nu^{-r\ell}_{\eta}).
\]
This form satisfies \eqref{THATBLUEHEDGEHOG}.
Since $\ord_{\infty}(G_{r,\ell}) \geq \frac{r\ell+1}{24}$ and $\ord_{\infty}(G_{r,\ell}) \leq \frac{k}{12}$, it follows that $k \geq \frac{r\ell+1}{2}>\frac{r-1}{2}(\ell-1)$. By \eqref{after mod l -1}, we get $k=\frac{r+1}{2}(\ell-1)$; the result follows.
\end{proof}
We now construct the half-integral weight forms which we will need in order to prove the second assertion of Theorem~\ref{thm:arbitraryweight1}.
\begin{proposition}\label{arbitraryweight2proposition}
Let $r$ be an odd positive integer. Let $\ell \geq 5$ be a prime such that  $r \leq \ell-4$.
Suppose that all of the integers $-r \leq w<0$ satisfying $w \equiv -r \pmod{24}$ have the property that $\(\frac{-rw}{\ell}\)=1$.
Then there exists a form $g_{r,\ell}$ in $S_{\frac{\ell^{2}-r-1}{2}}(1, \nu^{-r}_{\eta},\Z)$ satisfying
\begin{equation}\label{OFALLPLACES}
g_{r,\ell} \equiv \sum_{\(\frac{-rn}{\ell}\)=-1}p_{r}\(\frac{n+r}{24}\)q^{\frac{n}{24}}.
\end{equation}
\end{proposition}
\begin{proof}
By $(3)$ of Lemma~\ref{Serre}, there exists a form $H_{r,\ell} \in S_{\frac{r(\ell^{2}-1)}{2}+(\ell+1)\frac{\ell-1}{2}}(1, \Z)$ with 
\[
H_{r,\ell} \equiv \(\frac{-r}{\ell}\)\Theta^{\frac{\ell-1}{2}}\Delta^{r\delta_{\ell}} \pmod{\ell}.
\]
Recall that $E_{\ell-1} \equiv 1 \pmod{\ell}$. Define $\widehat{H}_{r,\ell} \in S_{\frac{r(\ell^{2}-1)}{2}+(\ell+1)\frac{\ell-1}{2}}(1)$ by
\[
\widehat{H}_{r,\ell}:=\Delta^{r\delta_{\ell}}E^{\frac{\ell+1}{2}}_{\ell-1}-H_{r,\ell}.
\]
By \eqref{relevantdeltaformula}, we have
\[
\widehat{H}_{r,\ell} \equiv \sum \(1-\(\frac{-rn}{\ell}\)\)p_{r}(n-r\delta_{\ell})q^{n}\cdot\prod^{\infty}_{n=1}\(1-q^{\ell n}\)^{r\ell} \pmod{\ell}.
\]
This implies that
\[
\widehat{H}_{r,\ell} \equiv \eta^{r\ell^{2}}\(\sum_{\(\frac{-rn}{\ell}\)=0}p_{r}\(\frac{n+r}{24}\)q^{\frac{n}{24}}+2\sum_{\(\frac{-rn}{\ell}\)=-1}p_{r}\(\frac{n+r}{24}\)q^{\frac{n}{24}}\) \pmod{\ell}.
\]
 Since $\(\frac{-rw}{\ell}\)=1$ for all of the integers $-r \leq w<0$ which satisfy $w \equiv~-r~\pmod{24}$, we conclude that $\widehat{H}_{r,\ell}$ vanishes to order $>\frac{r\ell^{2}}{24}$ modulo $\ell$. By arguing as in the proof of Proposition~\ref{FILTRATION} using \eqref{eqn:INTEGRALITY}, we may assume that  $\widehat{H}_{r,\ell}$ vanishes to order $>\frac{r\ell^{2}}{24}$, so
 \[
 \frac{\widehat{H}_{r,\ell}}{\eta^{r\ell^{2}}} \in S_{\frac{\ell^{2}-r-1}{2}}(1,\nu^{-r}_{\eta},\Z).
 \]

We have a form $f_{r,\ell} \in S_{\frac{\ell-r-1}{2}}(1,\nu^{-r\ell}_{\eta})$ which satisfies \eqref{THATBLUEHEDGEHOG}, so the form
$f^{\ell}_{r,\ell} \in S_{\frac{\ell^{2}-r\ell-\ell}{2}}(1,\nu^{-r}_{\eta})$ has the property that

\[
f^{\ell}_{r,\ell} \equiv \sum_{\(\frac{-rn}{\ell}\)=0}p_{r}\(\frac{n+r}{24}\)q^{\frac{n}{24}} \pmod{\ell}.
\]
Let $s \in \Z$ satisfy $s \equiv 2^{-1} \pmod{\ell}$, and define $g_{r,\ell} \in S_{\frac{\ell^{2}-r-1}{2}}(1,\nu^{-r}_{\eta},\Z)$ by
\[
g_{r,\ell}:= s\(\frac{\widehat{H}_{r,\ell}}{\eta^{r\ell^{2}}}-f^{\ell}_{r,\ell}E^{\frac{r+1}{2}}_{\ell-1}\) \equiv \sum_{\(\frac{-rn}{\ell}\)=-1}p\(\frac{n+r}{24}\)q^{\frac{n}{24}} \pmod{\ell}.
\]
This concludes the proof.
\end{proof}

\section{Compatibility}
Suppose that $\ell \geq 5$ is prime and that $r$ is an odd positive integer such that $r < \ell-4$.
 Recall that 
 \[
N_{r}= \begin{cases} 
2 & \text{if } 3 \mid r, \\
6  & \text{if }  3\nmid r
\end{cases}
\] 
and that
  \[
\epsilon_{r}= \begin{cases} 
\(\frac{8}{r\ell}\) & \text{if } 3 \mid r, \\
(-\(\frac{8}{r\ell}\),-\(\frac{12}{r\ell}\) ) & \text{if }  3\nmid r.
\end{cases}
\] 
Recall also that $r$ and $\ell$ are \emph{compatible} if
for every newform $f \in S^{\operatorname{new}}_{\ell-r-2}(N_{r},\epsilon_{r})$, the image of $\bar{\rho}_{f}$ contains a conjugate of $\SL_{2}(\mathbb{F}_{\ell})$.
Finally, recall for each newform $f=\sum a(n)q^{n}$ that we have a mod $\ell$ reduction $\bar{f}=\sum \bar{a(n)}q^{n} \in \bar{\mathbb{F}}_{\ell}[[q]]$.
We now show that we have compatibility for a wide range of $r$ and $\ell$.
\begin{proposition}\label{Corollary1}
If $2^{r+2} \not \equiv 2^{\pm 1} \pmod{\ell}$ and $\ell > 5r+19$, then $r$ and $\ell$ are compatible. 
\end{proposition}
 \begin{proof}
  Let $f=q+\sum_{n \geq 2}a(n)q^{n} \in S^{\text{new}}_{\ell-r-2}(N_{r},\epsilon_{r})$ be a newform.
 By \cite[Theorem $2.47(b)$]{DDT97}, there are four possibilities for the image of $\bar{\rho}_{f}$:
 \begin{enumerate}
 \item
 $\bar{\rho}_{f}$ is reducible.
 \item
 $\bar{\rho}_{f}$ is dihedral, i.e. $\bar{\rho}_{f}$ is irreducible but $\bar{\rho}_{f}|_{G_{K}}$ is reducible for some quadratic $K/\Q$.
 \item
 $\bar{\rho}_{f}$ is exceptional, i.e. the projective image of $\bar{\rho}_{f}$ is conjugate to one of $A_{4}$, $S_{4}$, $A_{5}$.
 \item
 The image of $\bar{\rho}_{f}$ contains a conjugate of $\SL_2(\mathbb{F}_{\ell})$.
 \end{enumerate}
 Since $2^{r+2} \not \equiv 2^{\pm 1} \pmod{\ell}$, we have $2^{\ell-r-3} \not \equiv 2^{\pm 1} \pmod{\ell}$. 
 By \cite[Lemma 3.2]{AAT}, we conclude that $\bar{\rho}_{f}$ is irreducible. 
 By the condition $\ell >5r+19$, we have $\ell-r-2 \neq \frac{\ell+1}{2}, \frac{\ell+3}{2} $; by the same lemma, we see that $\bar{\rho}_{f}$ is not dihedral.

 Thus, in order to show that $r$ and $\ell$ are compatible, it suffices to show that the image of $\bar{\rho}_{f}$ cannot be exceptional. To this end, we show that the projective image contains an element of order $> 5$.
 Suppose that $\iota_{\ell}(a_{\ell})$ is an $\ell$-adic unit. Recall that $\chi$ is the $\ell$-adic cyclotomic character.
 By part $3$ of Theorem~\ref{bigGalThm}, we know that 
 \[
\rho_{f}|_{I_{\ell}} \cong \left(\begin{matrix}\chi^{\ell-r-3} & * \\0 & 1 \end{matrix}\right).
\]
Recall that $\omega$ is the mod $\ell$ cyclotomic character. Since $\omega$ has order $\ell-1$, we see that the projective image of $\bar{\rho}_{f}$ 
contains an element of order $\geq \frac{\ell-1}{\text{gcd}(\ell-1,\ell-r-3)}$. Since $\ell > 5r+19$, we see that $\frac{\ell-1}{\text{gcd}(\ell-1,\ell-r-3)} > 5$.

Recall that $\omega_{2}$ and $\omega_{2}'$ are Serre's fundamental characters of level $2$.
If $\iota_{\ell}(a_{\ell})$ is not an $\ell$-adic unit, then part $3$ of Theorem~\ref{bigGalThm} implies that
 \[
\bar{\rho}_{f}|_{I_{\ell}} \cong \left(\begin{matrix}\omega_{2}^{\ell-r-3} & 0 \\0 & \omega_{2}'^{\ell-r-3} \end{matrix}\right).
\]
Since $\omega_{2}/\omega^{'}_{2}$ has order $\ell+1$, we know that the projective image of $\bar{\rho}_{f}$ 
contains an element of order $\frac{\ell+1}{\text{gcd}(\ell+1,\ell-r-3)}$. By the fact that $\ell > 5r+19$, we see that $\frac{\ell+1}{\text{gcd}(\ell+1,\ell-r-3)} > 5$.
 \end{proof}
By computing in Magma, for small $r$ we can establish compatibility for a wide range of $\ell$.
 \begin{proposition}\label{Corollary2}
 Let $r$ be an odd positive integer satisfying $1 \leq r \leq 39$.
 \begin{enumerate}
 \item
If $\ell > r+4$ is a prime such that $2^{r+2} \not \equiv 2^{\pm 1} \pmod{\ell}$, then $r$ and $\ell$ are compatible.
\item
In particular, if $r=1$ or $3$ and $\ell > r+4$ is prime, then $r$ and $\ell$ are always compatible.
\end{enumerate}
\end{proposition}
 \begin{proof}
 Fix an odd integer $r$ such that $1 \leq r \leq 39$. Our strategy to prove $(1)$ is to compute in Magma to rule out each of the first three possibilities in 
Proposition~\ref{Corollary1}
  as we vary over the primes $\ell > r+4$ such that $2^{r+2} \not \equiv 2^{\pm 1} \pmod{\ell}$.
 By the same proposition, we may assume that $r+4 < \ell \leq 5r+19$.
 Fix such a prime $\ell$. 
Arguing as in the proof of Proposition~\ref{Corollary1}, we conclude that $\bar{\rho}_{f}$ is irreducible for each newform $f \in S^{\text{new}}_{\ell-r-2}(N_{r},\epsilon_{r})$.

For each newform $f=q+\sum_{n \geq 2}a_{f}(n)q^{n} \in S^{\text{new}}_{\ell-r-2}(N_{r},\epsilon_{r})$,
 by \cite[Lemma 3.2]{AAT}, we see that $\bar{\rho}_{f}$ is dihedral only if $\bar{\rho}_{f} \cong \bar{\rho}_{f} \otimes \omega^{\frac{\ell-1}{2}}$ and  $\ell=2r+5$ or $\ell=2r+7$. 
 If $p \neq \ell$ is prime, then $\chi(\operatorname{Frob}_{p})=p$. 
if $\ell=2r+5$ or $2r+7$, then we compute in Magma to conclude that there exists $p \in \{5,7\}$ with the properties that $p \neq \ell$, $\omega^{\frac{\ell-1}{2}}(\operatorname{Frob}_{p})=-1$ (which would mean that $\text{tr }\bar{\rho}_{f}(\operatorname{Frob}_{p})=0$) and $\bar{a_{f}(p)} \neq 0$. 
 However, this contradicts the fact that $\text{tr }\bar{\rho}_{f}(\operatorname{Frob}_{p})= \bar{a_{f}(p)}$. Therefore, $\bar{\rho}_{f}$ is not dihedral for each newform $f \in S^{\text{new}}_{\ell-r-2}(N_{r},\epsilon_{r})$.

Thus, to prove $(1)$, it suffices to show for each newform $f \in S^{\text{new}}_{\ell-r-2}(N_{r},\epsilon_{r})$ that $\bar{\rho}_{f}$ does not have exceptional image. From the proof of Proposition~\ref{Corollary1}, we know that this condition holds if 
 $\frac{\ell-1}{\text{gcd}(\ell-1,\ell-r-3)}, \frac{\ell+1}{\text{gcd}(\ell+1,\ell-r-3)} > 5.$ 
If either of these inequalities fails to hold, then let $p \nmid 6\ell$ be a prime. Define $u_{f}(p):= \bar{a_{f}(p)^{2}}/p^{\ell-r-3}$. 
If the projective image of $\bar{\rho}_{f}$ is $A_{4}, S_{4}, \text{ or } A_{5}$, then 
we have
  \begin{equation}\label{THATBLUEHEDGEHOGAGAINOFALLPLACES!}
  u_{f}(p) = 0,1,2,4 \ \ \ \text{ or } u_{f}(p)^{2}-3u_{f}(p)+1 = 0,
  \end{equation}
  depending on the order of the image of $\bar{\rho}_{f}(\operatorname{Frob}_{p})$ in $\operatorname{PGL}_{2}(\bar{\mathbb{F}}_{\ell})$ (see e.g. \cite[p. 264]{therealslimshady} and \cite[p. 189]{Ribet}). To prove our result, it suffices to check for each newform $f \in S^{\text{new}}_{\ell-r-2}(N_{r},\epsilon_{r})$ that there exists a  prime $p \nmid 6\ell$ such that \eqref{THATBLUEHEDGEHOGAGAINOFALLPLACES!} does not hold.
 For such $f$, we compute with Magma to show that there exists $p \in \{5,7\}$ such that \eqref{THATBLUEHEDGEHOGAGAINOFALLPLACES!} does not hold for $u_{f}(p)$. 
 
 To prove $(2)$, note that if $r=1$ or $3$, then $2^{r+2} \not \equiv 2^{\pm 1} \pmod{\ell}$ only fails when $\ell=5$ or $7$, respectively; this contradicts the assumption that $\ell > r+4$.
 \end{proof}
 \section{The main technical result}
  Choose a number field $E$ containing all of the coefficients of all of the normalized Hecke eigenforms in $S^{\text{new}}_{\ell-r-2}(N_{r},\epsilon_{r})$. Recall that for each prime $\ell \geq 5$ we have fixed an embedding $\iota_{\ell}:\bar{\Q} \rightarrow \bar{\Q}_{\ell}$. If $\lambda$ is the prime of $E$ induced by the embedding $\iota_{\ell}: \bar{\Q} \rightarrow \bar{\Q}_{\ell}$, then let $E_{\lambda}$ be the completion of $E$ at $\lambda$ with ring of integers $\mathcal{O}_{\lambda}$. For a fixed $\alpha \in \Z$, we can assume without loss of generality that $\mathcal{O}_{\lambda}$ has the property that
  \begin{equation}\label{Annoying}
   \text{ the polynomial $x^2-\alpha x+1$ factors in $\mathcal{O}_{\lambda}$ with roots $\alpha_{1}$ and $\alpha_{2}$}.
   \end{equation}
   
 We now prove the main technical result which we will need in order to prove Theorem~\ref{thm:main technical}.
 
 \begin{theorem}\label{theoremtechnical}
 Let $\ell \geq 5$ be prime. Suppose that $\lambda$, $\alpha$, $\alpha_{1}$, and $\alpha_{2}$ are defined as above and that $\alpha \not \equiv \pm 2 \pmod{\ell}$. 
 Let $m \geq 1$ be an integer. 
 Suppose that $r$ and $\ell$ are compatible. 
 Then there exists a positive density set $S$ of primes such that if $Q \in S$, then $Q \equiv 1 \pmod{\ell^m}$ and for all of the newforms $f \in S^{\operatorname{new}}_{\ell-r-2}(N_{r},\epsilon_{r})$, we have
  \[
  f \sl T_{Q} \equiv (\alpha_{1}^{\ell^{m-1}}+\alpha_{2}^{\ell^{m-1}})f \pmod{\lambda^m}.
  \]
 \end{theorem}
 \begin{proof}
 By the assumption that $\alpha \not \equiv \pm 2 \pmod{\ell}$, we know that $\alpha_{1} \not \equiv \alpha_{2} \pmod{\lambda}$. 
 For each newform $f \in S^{\operatorname{new}}_{\ell-r-2}(N_{r},\epsilon_{r})$, the Galois representations $\rho_{f}$ and $\bar{\rho}_{f}$ can be defined over $\mathcal{O}_{\lambda}$ and $\mathcal{O}_{\lambda}/\lambda$, respectively. Define $\bar{\alpha}:= \alpha \pmod{\ell}$.
Since $r$ and $\ell$ are compatible, we know that the image of each $\bar{\rho}_{f}$ contains a conjugate of $\SL_2(\mathbb{F}_{\ell})$, 
so we can use \cite[~Proposition $3.8$]{AAT} to find an element $\sigma \in \Gal(\bar{\Q}/\Q(\zeta_{\ell}))$ such that each $\bar{\rho}_{f}(\sigma)$ is conjugate to $\left(\begin{matrix}\bar{\alpha} & 1 \\-1 & 0\end{matrix}\right)$. 
It follows that the characteristic polynomial of $\rho_{f}(\sigma)$ is congruent to $x^2-\alpha x+1$ modulo $\lambda$. 
 Since
 $\alpha_1 \not \equiv \alpha_{2} \pmod{\lambda}$, 
 we can use Hensel's lemma to factor the characteristic polynomial of $\rho_{f}(\sigma)$ over $\mathcal{O}_{\lambda}$.
 Thus, $\rho_{f}(\sigma)$ is conjugate to a diagonal matrix with entries $\beta_{1}$ and $\beta_{2}$ in $\mathcal{O}_{\lambda}$ with the properties that $\beta_{1} \equiv \alpha_{1} \pmod{\lambda}$ and $\beta_{2} \equiv \alpha_{2} \pmod{\lambda}$. 
By induction, we see that $\beta^{\ell^{m-1}}_{1} \equiv \alpha^{\ell^{m-1}}_{1} \pmod{\lambda^m}$ and $\beta^{\ell^{m-1}}_{2} \equiv \alpha^{\ell^{m-1}}_{2} \pmod{\lambda^m}$. 
 This implies that the characteristic polynomial of $\rho_{f}(\sigma^{\ell^{m-1}})$ is congruent to $x^{2}-(\alpha^{\ell^{m-1}}_{1}+\alpha^{\ell^{m-1}}_{2})x+1 \pmod{\lambda^m}.$

 Since $\Gal(\Q(\zeta_{\ell^{m}})/\Q(\zeta_{\ell}))$ has order $\ell^{m-1}$, we have $\sigma^{\ell^{m-1}} \in \Gal(\bar{\Q}/\Q(\zeta_{\ell^{m}}))$.
 By composing $\rho_{f} |_{\Gal(\bar{\Q}/\Q(\zeta_{\ell^{m}}))}$ with the projection $\tau_{m}:\GL_2(\mathcal{O}_{\lambda}) \rightarrow \GL_2(\mathcal{O}_{\lambda}/\lambda^{m})$, we have a representation $\bar{\rho}_{f,m}:\Gal(\bar{\Q}/\Q(\zeta_{\ell^{m}})) \rightarrow \GL_2(\mathcal{O}_{\lambda}/\lambda^{m})$. 
 This representation has finite image, so there exists a finite extension $L/\Q(\zeta_{\ell^{m}})$ such that if $\tau:\Gal(\bar{\Q}/\Q(\zeta_{\ell^m})) \rightarrow \Gal(L/\Q(\zeta_{\ell^{m}}))$ is the restriction map, then there exists a representation $\widetilde{\rho}_{f,m}:\Gal(L/\Q(\zeta_{\ell^{m}})) \rightarrow \GL_2(\mathcal{O}_{\lambda}/\lambda^{m})$ such that $\widetilde{\rho}_{f,m} \circ \tau= \bar{\rho}_{f,m}$. 
 The Chebotarev density theorem implies that there is a positive density set $S$ of primes such that if $Q \in S$, then $\widetilde{\rho}_{f,m}(\text{Frob}_{Q} |_{L})$ is conjugate to $\widetilde{\rho}_{f,m}(\sigma^{\ell^{m-1}} \sl_{L})$; the fact that $\text{Frob}_{Q} |_{L} \subset \Gal(L/\Q(\zeta_{\ell^{m}}))$ for such $Q$ implies that $Q \equiv 1 \pmod{\ell^{m}}$.
 For such $Q$ and any newform $f \in S^{\text{new}}_{\ell-r-2}(N_{r},\epsilon_{r})$, we have
 \[
 f \sl T_{Q} \equiv (\text{tr }\widetilde{\rho}_{f,m}(\text{Frob}_{Q} |_{L}))f \equiv (\alpha^{\ell^{m-1}}_{1}+\alpha_{2}^{\ell^{m-1}})f \pmod{\lambda^{m}}.
 \]
 \end{proof}
 \begin{remark}
 We assume that $\alpha \not \equiv \pm 2 \pmod{\ell}$ so that we can apply Hensel's lemma. However, we do not need Hensel's lemma when $m=1$. Thus, 
 if $m=1$, then we have the result when $\alpha \equiv  \pm 2\pmod{\ell}$.
 \end{remark}
 \section{Proofs of Theorems $1.2$ and $1.3$}
 We now prove Theorems $1.2-1.3$. First, we introduce some notation.
 Let $\ell \geq 5$ be prime. Suppose that $f \in S_{k}(1,\nu^{r}_{\eta},\Z)$  with $f \not \equiv 0 \pmod{\ell}$, where $r$ is an odd positive integer and $k \in \frac{1}{2}\Z\backslash \Z$ satisfies $k \geq \frac{5}{2}$.
 We define
 \[
 S(r,k):= \begin{cases} 
S^{\operatorname{new}}_{2k-1}\(2,\(\frac{8}{r}\)\) & \text{if } 3 \mid r, \\
S^{\operatorname{new}}_{2k-1}\(6,-\(\frac{8}{r}\),-\(\frac{12}{r}\)\) & \text{if }  3\nmid r
\end{cases}
\] 
and (as in the introduction)
\[
 \chi^{(r)}:= \begin{cases} 
\(\frac{-4}{\bullet}\) & \text{if } 3 \mid r, \\
\(\frac{12}{\bullet}\)& \text{if }  3\nmid r.
\end{cases}
\] 
 For each squarefree $t$, let
$F_{t} \in S(r,k)$ 
be the form with $\text{Sh}_{t}f=F_{t} \otimes \chi^{(r)}$.
As $t$ ranges over all of the squarefree integers, there are only finitely many non-zero possibilities for $F_{t} \pmod{\ell}$. Let $S:=\{F_{t_{1}},...,F_{t_{k}}\}$ be a collection which represents all of these possibilities. By \eqref{f0iffShimura0}, we see that $S$ is not empty.

The space $S(r,k)$  is spanned by newforms $g_{1},...,g_{d}$.
 For $j \in \{1,...,k\}$, write
 \[
 F_{t_{j}}=\sum^{d}_{i=1}c_{i,j}g_i,
 \]
 and let $E$ be a number field which contains the coefficients of each $g_i$ as well as all of the coefficients $c_{i,j}$. Fix a prime $\lambda$ of $E$ over $\ell$ and define
 \[
 m(f):= \text{max}(1,1-\text{min}(\ord_{\lambda}(c_{i,j})))
 \]
 (the definition depends on the choice of $\lambda$, but this is not important to us).
 Before we prove Theorem~\ref{thm:main technical}, we require the following lemma.
 \begin{lemma}\label{AAT6-1}
 Suppose that $\ell \geq 5$ is prime and that $f \in S_{k}(1,\nu^{r}_{\eta},\Z)$, where $r$ is an odd positive integer.
Suppose that the newforms $g_i$ and the integer $m(f)$ are defined as above. 
  Suppose that $Q \geq 5$ is a prime and that $\lambda_{Q}$ is an algebraic integer in $E$ with the property that
  $g_{i} \sl T_{Q} \equiv \lambda_{Q}g_{i} \pmod{\lambda^{m(f)}}$ for all $i$. Then 
 \[
 f \sl T_{Q^{2}} \equiv \chi^{(r)}(Q)\lambda_{Q}f \pmod{\lambda}.
 \]
 \end{lemma}
 
 \begin{proof}
 For each $t_{j}$ we have 
 \[
 F_{t_{j}} \otimes \chi^{(r)}= \sum^{d}_{i=1} c_{i,j}g_{i} \otimes \chi^{(r)},
 \]
 and for each $i$ we have
 \[
 \(g_{i} \otimes \chi^{(r)}\) \sl T_{Q}=\chi^{(r)}(Q)\(g_{i}\sl T_{Q}\) \otimes \chi^{(r)} \equiv \chi^{(r)}(Q)\lambda_{Q}g_{i} \otimes \chi^{(r)} \pmod{\lambda^{m(f)}}.
 \]
 By the definition of $m(f)$, it follows for each $t_{j}$ that
 \[
 \(F_{t_{j}} \otimes \chi^{(r)}\) \sl T_{Q} \equiv \chi^{(r)}(Q) \lambda_{Q}F_{t_{j}} \otimes \chi^{(r)} \pmod{\lambda}.
 \]
 Thus, for each squarefree $t$, we have
 \[
\Sh_{t}\(f\sl T_{Q^{2}}\)=\(\Sh_{t}f\)\sl T_{Q}  \equiv \chi^{(r)}(Q)\lambda_{Q}\Sh_{t}f \pmod{\lambda}.
 \]
 The result follows from \eqref{f0iffShimura0}.
 \end{proof}
The next result explains how to produce congruences from Lemma~\ref{AAT6-1}.
 \begin{lemma}\label{AAT6-2}
 Suppose that $\ell \geq 5$ is prime, that $r$ is an odd positive integer, and that $f=\sum a(n)q^{\frac{n}{24}} \in S_{k}(1,\nu^{r}_{\eta},\Z)$.
 Suppose that $Q \geq 5$ is prime.
 Suppose that there exists $\alpha_{Q} \in \{\pm 1\}$ and $\alpha \in \Z$ with
 \[
 f \sl T_{Q^{2}} \equiv \alpha_{Q}\alpha Q^{k-\frac{3}{2}}f \pmod{\ell}.
 \]
 Then we have
 \[
 a(Q^{2}n) \equiv (\alpha-1)\alpha_{Q}Q^{k-\frac{3}{2}}a(n) \pmod{\ell} \ \ \ \text{ if } \ \ \(\frac{n}{Q}\)=
 \alpha_{Q}\(\frac{12}{Q}\)\(\frac{-1}{Q}\)^{k-\frac{1}{2}}.
 \]
 \end{lemma}
 \begin{proof}
 This follows from the definition of the Hecke operator in \eqref{Hecke1}. If $f \sl T_{Q^{2}}=\sum b(n)q^{\frac{n}{24}}$, then the third term of $b(n)$ does not contribute and if $\(\frac{n}{Q}\)=
 \alpha_{Q}\(\frac{12}{Q}\)\(\frac{-1}{Q}\)^{k-\frac{1}{2}}$, then the middle term of $b(n)$ becomes $\alpha_{Q}Q^{k-\frac{3}{2}}a(n)$.
 \end{proof}
 We now prove Theorem~\ref{thm:main technical}.
 \begin{proof}[Proof of Theorem~\ref{thm:main technical}]
 Suppose that $\ell \geq 5$ is prime, that $r < \ell-4$ is an odd positive integer and that $r$ and $\ell$ are compatible. 
By Proposition~\ref{FILTRATION}, 
there is a modular form $f_{r,\ell} \in S_{\frac{\ell-r-1}{2}}(1,\nu^{-r\ell}_{\eta})$ such that
 \[
 f_{r,\ell} \equiv \sum p_{r}\(\frac{\ell n+r}{24}\)q^{\frac{n}{24}} \pmod{\ell}.
 \]
 Since $\chi^{(-r\ell)}=\chi^{(r)}$, each of its Shimura lifts lands in the space $S\(-r\ell,\frac{\ell-r-1}{2}\) \otimes \chi^{(r)}$.

For each squarefree $t$, let $\tilde{F}_{t} \in S(-r\ell,\frac{\ell-r-1}{2})$ be the form with $\text{Sh}_{t}f_{r,\ell}=\tilde{F}_{t} \otimes \chi^{(r)}$.
 By \cite[Exercise $6.4$]{Serre2}, we see that there exists a positive density set $S$ of primes such that if $Q \in S$, then  $Q \equiv 1 \pmod{\ell}$ and
 \[
\tilde{F}_{t}
\sl T_{Q} \equiv 2
\tilde{F}_{t}
\pmod{\ell}.
\]
for all squarefree $t$.
Thus, for all squarefree $t$, we have
 \[
 \(\tilde{F}_{t} \otimes \chi^{(r)}\) \sl T_{Q} \equiv 2\chi^{(r)}(Q)\tilde{F}_{t} \otimes \chi^{(r)} \pmod{\ell},
 \]
 which means that 
  \[
\Sh_{t}\(f\sl T_{Q^{2}}\)=\(\Sh_{t}f\)\sl T_{Q}  \equiv 2\chi^{(r)}(Q)\Sh_{t}f \pmod{\ell}.
 \]
 It follows from \eqref{f0iffShimura0} that
\[
f_{r,\ell}\sl T_{Q^{2}} \equiv 2\chi^{(r)}(Q)f_{r,\ell} \pmod{\ell}.
\]
 Thus, for $\alpha \equiv 2 \pmod{\ell}$, the result follows from Lemma~\ref{AAT6-2}.

Recall the definitions of $\alpha_{1}$ and $\alpha_{2}$ in \eqref{Annoying}.
If $\alpha \not \equiv \pm 2 \pmod{\ell}$, then Theorem~\ref{theoremtechnical}  implies that there is a positive density set $S$ of primes such that if $Q \in S$, then $Q \equiv 1 \pmod{\ell}$ and for all of the newforms $g_i$ in $S(-r\ell,\frac{\ell-r-1}{2})$, we have
 \begin{equation}\label{Ionlycitethisonce}
 g_{i} \sl T_{Q} \equiv (\alpha^{\ell^{m(f_{r,\ell})-1}}_1+\alpha_{2}^{\ell^{m(f_{r,\ell})-1}})g_{i} \pmod{\lambda^{m(f_{r,\ell})}}.
 \end{equation}
 Since
 \[
 (\alpha^{\ell^{m(f_{r,\ell})-1}}_1+\alpha_{2}^{\ell^{m(f_{r,\ell})-1}}) \equiv (\alpha_{1}+\alpha_{2})^{\ell^{m(f_{r,\ell})-1}} \equiv \alpha^{\ell^{m(f_{r,\ell})-1}} \equiv \alpha \pmod{\lambda},
 \]
it follows from Lemma~\ref{AAT6-1} that  $f_{r,\ell} \sl T_{Q^{2}} 
 \equiv \alpha\chi^{(r)}(Q) f_{r,\ell} \pmod{\lambda}$. Therefore,
 \[
 f_{r,\ell} \sl T_{Q^{2}} 
 \equiv \alpha\chi^{(r)}(Q) f_{r,\ell} \pmod{\ell},
 \]
and the result follows from Lemma~\ref{AAT6-2}.
 \end{proof}
 \begin{remark}
 If there are no congruences between distinct newforms in $S(-r\ell,\frac{\ell-r-1}{2})$, then $m(f_{r,\ell})=1$ by \cite[Lemma $5.3$]{AAT}. By the remark following the proof of Theorem~\ref{theoremtechnical}, this implies that \eqref{Ionlycitethisonce} holds when $\alpha \equiv -2 \pmod{\ell}$, which gives Theorem~\ref{thm:main technical} for this case.
 \end{remark}
 To prove Theorem~\ref{thm:arbitraryweight1},
 we make use of the following consequence of \cite[Thm. $4.2$]{AAT}.
 They only state the result for $f \in S^{\operatorname{new}}_{k}(6)$, but the same argument applies with $S^{\operatorname{new}}_{k}(6)$ replaced by $S^{\operatorname{new}}_{k}(2)$.
\begin{theorem}\label{Theorem4.2}
 Suppose that $\ell \geq 5$ is prime, that $r$ is an odd positive integer, and that there exists an integer $a$ for which $2^{a} \equiv -1 \pmod{\ell}$. 
Let $k \in \frac{1}{2}\Z \backslash \Z$. 
Then there exists a positive density set $S$ of primes such that if $Q \in S$, then $Q \equiv -2 \pmod{\ell}$ and for every $f \in S(r,k)$ 
we have $f \sl T_{Q} \equiv -(-\epsilon_{2})^{a}Q^{k-\frac{3}{2}}f \pmod{\ell}$, where $\epsilon_{2}$ is the $W_{2}$ eigenvalue of $f$.
\end{theorem}
 \begin{proof}[Proof of Theorem~\ref{thm:arbitraryweight1}]
 Suppose that $r$ is an odd positive integer. Suppose that $\ell \geq 5$ is prime, that $r < \ell-4$ and that there exists an integer $a$ with $2^{a} \equiv -1 \pmod{\ell}$. For each squarefree $t$, let $\tilde{F}_{t} \in S(-r\ell,\frac{\ell-r-1}{2})$ be the form with $\text{Sh}_{t}f_{r,\ell}=\tilde{F}_{t} \otimes \chi^{(r)}$.
 By Theorem~\ref{Theorem4.2}, we see that there exists $\gamma \in \{\pm 1\}$ and a positive density set $S$ of primes such that if $Q \in S$, then  $Q \equiv -2 \pmod{\ell}$ and
 \[
\tilde{F}_{t}
\sl T_{Q} \equiv \gamma Q^{\frac{\ell-r-4}{2}}
\tilde{F}_{t}
\pmod{\ell}.
\]
for all $t$.
Thus, for all $t$, we have
 \[
 \(\tilde{F}_{t} \otimes \chi^{(r)}\) \sl T_{Q} \equiv \gamma \chi^{(r)}(Q) Q^{\frac{\ell-r-4}{2}}\tilde{F}_{t} \otimes \chi^{(r)} \pmod{\ell},
 \]
 which means that 
  \[
\Sh_{t}\(f_{r,\ell}\sl T_{Q^{2}}\)=\(\Sh_{t}f_{r,\ell}\)\sl T_{Q}  \equiv \gamma \chi^{(r)}(Q) Q^{\frac{\ell-r-4}{2}} \Sh_{t}f_{r,\ell} \pmod{\ell}.
 \]
 It follows from \eqref{f0iffShimura0} that
\[
f_{r,\ell}\sl T_{Q^{2}} \equiv \gamma \chi^{(r)}(Q)Q^{\frac{\ell-r-4}{2}}f_{r,\ell} \pmod{\ell}.
\]
It then follows (after setting $\alpha_{Q}=\gamma \chi^{(r)}(Q)$ and $\beta=1$) from Lemma~\ref{AAT6-2} 
that there exists $\epsilon_{Q} \in \{\pm 1\}$ with the property that
\[
p_{r}\( \frac{\ell Q^{2}n+r}{24}\) \equiv 0 \pmod{\ell} \ \ \ \text{ if } \ \ \(\frac{n}{Q}\)= \epsilon_{Q}.
\]
This proves part $(1)$ of Theorem~\ref{thm:arbitraryweight1}.

We now prove part $(2)$.
Under the hypotheses on the integers $-r \leq w <0$ which satisfy $w \equiv -r \pmod{24}$,
we have the form $g_{r,\ell}=\sum a(n)q^{\frac{n}{24}} \in S_{\frac{\ell^{2}-r-1}{2}}(1,\nu^{-r}_{\eta},\Z)$ constructed in Proposition~\ref{arbitraryweight2proposition}. It satisfies
\[
a(n) \equiv p_{r}\(\frac{n+r}{24}\) \pmod{\ell} \ \ \ \ \text{ when } \(\frac{-rn}{\ell}\)=-1.
\]
We now argue as in the proof of part $(1)$ to conclude that there exists $\epsilon_{Q} \in \{\pm 1\}$ for which
\[
p_{r}\(\frac{Q^{2}n+r}{24}\) \equiv 0 \pmod{\ell} \ \ \ \ \text{if} \ \ \(\frac{-rn}{\ell}\)=-1 \ \ \text{ and } \ \ \ \(\frac{n}{Q}\)=\epsilon_{Q} .
\]
The result follows.
 \end{proof}
 
 \section{The case when $r=\ell-4$}
 When $r=\ell-4$, we obtain results stronger than Theorems $1.1-1.3$. The following result gives a short proof of \cite[Thm. $2.1$, $(2)$]{Boylan}.
\begin{lemma}
If $\ell \not \equiv 1 \pmod{6}$ is prime, then $f_{\ell-4,\ell}=0$. Thus, we have 
\[
p_{\ell-4}\(\frac{\ell n+\ell-4}{24}\)\equiv 0 \pmod{\ell} \ \ \ \text{for all $n$} .
\]
\end{lemma}
\begin{proof}
Note that
\begin{equation}\label{had to tag this}
f_{\ell-4,\ell} \in S_{\frac{3}{2}}(1,\nu^{4\ell-1}_{\eta}).
\end{equation}
Let $s \in \Z$ be such that $s \equiv 4\ell-1 \pmod{24}$ and $0<s<24$. By \eqref{rmod24expansion}, we see that $f_{\ell-4,\ell}$ vanishes to order $\geq s$. 
Since $\ell \not \equiv 1 \pmod{6}$, we have $s > 3$, so $\eta^{-3}f_{\ell-4,\ell} \in S_{0}(1,\nu_{\eta}^{4\ell-4})=\{0\}$. The result follows.
\end{proof}
 
 When $\ell \equiv 1 \pmod{6}$, we can use the fact that $\eta^{3}$ is an eigenform for the Hecke operators $T_{Q^{2}}$ for all primes $Q \geq 3$ to produce the following congruences.
 \begin{lemma}
 If $\ell \equiv 1 \pmod{6}$ is prime, then $f_{\ell-4}=c \eta^{3}$ for some $c \in \Z$. For all primes $Q \geq 3$, if $Q^{2} \nmid n$, then
 \[
 p_{r}\(\frac{\ell Q^{2}n+r}{24}\) \equiv \(\frac{-1}{Q}\)\(Q+1-\(\frac{12n}{Q}\)\)p_{r}\(\frac{\ell n+r}{24}\) \pmod{\ell}.
 \]
 \end{lemma}
 
 \begin{proof}
 Since $\ell \equiv 1 \pmod{6}$, we know that $3 \mid 4\ell-1$.
We also see by \eqref{had to tag this} that 
\[
\eta^{-3}f_{\ell-4,\ell} \in M_{0}(1,\nu^{4\ell-4}_{\eta})=M_{0}(1). 
\]
Since $f_{\ell-4,\ell}$ has integral coefficients, we conclude that $f_{\ell-4,\ell}=c \eta^{3}$ for some $c \in \Z$.
It follows from \eqref{Hecke1} for each prime $Q \geq 3$ that  
 \[
 \eta^{3} \sl T_{Q^{2}}=\(\frac{-1}{Q}\)(Q+1)\eta^{3}. 
 \]
 Thus, when $Q^{2} \nmid n$, we have
 \[
 p_{r}\(\frac{\ell Q^{2} n+r}{24}\) \equiv \(\frac{-1}{Q}\)\(Q+1-\(\frac{12n}{Q}\)\)p_{r}\(\frac{\ell n+r}{24}\) \pmod{\ell}.
 \]
 \end{proof}
 \section{Acknowledgements}
The author would like to thank Scott Ahlgren for suggesting this project and making many helpful comments.
The author would also like to thank the Graduate
College Fellowship program at the University of Illinois at Urbana-Champaign and the Alfred P. Sloan
Foundation for their generous research support.
 \bibliographystyle{amsalpha}
\bibliography{Linearcongruencesforpowersofthepartitionfunction}
 \end{document}